\documentclass[12pt]{amsart}

\usepackage{graphicx}

\usepackage{hyperref}

\usepackage{xcolor}

\usepackage{mathtools, amsmath}

\usepackage{verbatim}

\usepackage[flushleft]{paralist}

\newtheorem{theorem}{Theorem}
\newtheorem{proposition}[theorem]{Proposition}
\newtheorem{lemma}[theorem]{Lemma}
\newtheorem{corollary}[theorem]{Corollary}
\newtheorem{conjecture}[theorem]{Conjecture}
\theoremstyle{definition}

\newtheorem{defn}[theorem]{Definition}
\newtheorem*{recap}{Theorem~\ref{THE theorem}}

\newcommand{\fix}{\operatorname{fix}}

\begin{document}

\title[Hurwitz Action on Factorizations in $G_6$]{Hurwitz Actions on Reflection Factorizations in Complex Reflection Group $G_6$}
\author[G. Gawankar]{Gaurav Gawankar$^1$}
\author[D. Lazreq]{Dounia Lazreq$^{1, *}$}
\author[M. Rai]{Mehr Rai$^1$}
\author[S. Sabar]{Seth Sabar$^2$}
\thanks{$^1$ George Washington University, Washington, DC, USA 20052}
\thanks{$^2$ School Without Walls, Washington, DC, USA 20052}
\thanks{$^*$ Corresponding author: \href{mailto:lazreqdounia@gwu.edu}{lazreqdounia@gwu.edu}}

\date{\today}

\maketitle

\begin{abstract}
We show that in the complex reflection group $G_6$, reflection factorizations of a Coxeter element that have the same length and multiset of conjugacy classes are in the same Hurwitz orbit. This confirms one case of a conjecture of Lewis and Reiner.
\end{abstract}

\section{Introduction}

The main result of this paper is the following theorem.

\begin{theorem}\label{THE theorem}
Let $T$ and $T'$ be two length-$n$ reflection factorizations of a Coxeter element of the complex reflection group $G_6$. Then, $T$ and $T'$ are in the same Hurwitz orbit if and only if they have the same multiset of conjugacy classes.
\end{theorem}

Theorem \ref{THE theorem} is a particular case of the following conjecture of Lewis and Reiner.

\begin{conjecture}[{\cite[Conj. 6.3]{LR}}]\label{L and R}
In a well-generated finite complex reflection group, two
reflection factorizations of a Coxeter element lie in the same Hurwitz orbit
if and only if they share the same multiset of conjugacy classes.
\end{conjecture}

This conjecture arose as a result of generalizing a theorem by Bessis {\cite[Prop. ~1.6.1]{Bessis}}, which is identical to Conjecture \ref{L and R} , but only makes the claim for shortest reflection factorizations of a Coxeter element. Conjecture \ref{L and R} was proven for real reflection groups by Lewis and Reiner \cite{LR}, for the groups $G_4$ and $G_5$ by Peterson \cite{Zach}, and for the infinite families of complex reflection groups by Lewis \cite{infinite}. 

As in the proofs of Peterson and Lewis--Reiner, we prove Theorem \ref{THE theorem} by induction. We begin in Section \ref{section2} by giving background information and defining important objects used. We first look at complex reflection groups as a whole and then consider the intricacies of $G_6$. In Section 3, we give the proof. We start by making some key observations about the outcomes of applying Hurwtiz moves and then begin constructing our inductive argument by using the idea of a marked element to move from one reflection factorization to one of a shorter length. By checking finite instances on Sage \cite{Sage}, we prove our base cases, and are able to fully construct our inductive argument, giving us a proof of Theorem~\ref{THE theorem}.

\subsection*{Acknowledgements}
We would like to thank Joel Lewis for his continued mentoring and support. Dounia Lazreq would also like to thank the Luther Rice Undergraduate Research Fellowship for supporting this project. 

\section{General Background}\label{section2}

\subsection{Complex Reflection Groups}

Let $V$ be a vector space over the field $\mathbb{C}$. Given a linear transformation $t:V\rightarrow V$, there are several subspaces of $V$ that can be found by considering how $t$ acts on $V$. One such subspace is the \textit{fixed space} of vectors that are unchanged when $t$ is applied, denoted by $\fix(t) = \{v\in V: t(v) = v\}$. We define $t$ to be a \textit{generalized reflection} if $\dim(\fix(t)) = \dim(V) -1$. In this case, $\fix(t)$ is called a \textit{reflecting hyperplane} of $t$. 

The objects that we are working with in this paper are \textit{complex reflection groups}. The group $G$ is defined to be a complex reflection group (CRG)  if it is a finite group of transformations $t:V\rightarrow V$, where there is a subset $P \subset G$ of reflections of $G$ such that every element of $G$ can be produced by multiplying together elements of $P$. Choosing an appropriate basis, we can also write a CRG $G$ as a finite group of $\dim(V)\times \dim(V)$ matrices with complex entries.

If we take some element $c$ of a CRG $G$, then the tuple $(r_1, r_2, \ldots , r_n)$ of reflections in $G$ is a \textit{reflection factorization} of $c$ if $c = r_1 \cdot r_2 \cdots r_n$, where $n$ is the \textit{length} of the factorization. If the elements $x$ and $y$ are both in the group $G$, and $x = qyq^{-1}$, where $q$ is also an element of $G$, then $x$ and $y$ are \textit{conjugate} to each other in the group. Notice that this divides the elements of the group into classes. A \textit{conjugacy class} of a CRG
is a set of elements of the group such that any two elements in the same  conjugacy class are conjugate to each other. 
 
 A complex reflection group that can't be written as
 $$G\times H = \left\{\left[\begin{array}{c|c} g & 0 \\ \hline 0 & h \end{array}\right] \textrm{where } g \in G \text{ and } h \in H\right\}$$ 
  and where $G$ and $H$ are themselves complex reflection groups, is called \textit{irreducible}. There is a classification of these irreducible groups that gives us a few infinite families and many exceptional cases \cite{ST}.
 
 Our theorem statement makes a claim about type of element that appear in some of these groups, called Coxeter elements. To understand what these are, we need to understand a few key properties of complex reflection groups. 
 
Consider a CRG $G$ acting irreducibly on a vector space $V$. The \textit{rank} of $G$ is the dimension of $V$, and let $n = \dim(V)$. Choosing an appropriate basis, we can write $G$ as a group of $n\times n$ matrices.  We say that $G$ is \textit{well-generated} if $G$ is of rank $n$ and there exists a set $P$ of exactly $n$ reflections such that these $n$ reflections generate all of $G$. That is, every element of $G$ is some product of elements of $P$. 

The \textit{order} of an element $r\in G$ is the smallest positive integer $m$ such that $r^m$ is the identity. If we are given an element $t$ of $G$ where $t$ is of order $k$ and $\lambda$ is an eigenvalue of $t$, then $\lambda$ is a $k$th \textit{root of unity}. That is, $\lambda^k = 1$. 

The element $g\in G$ is said to be \textit{Springer regular} if it has an eigenvector $v$ that does not lie in any of the reflecting hyperplanes of the reflections of $G$. If $g$ furthermore has order $k$, then we say that $g$ is \textit{$k$-regular}.

If a complex reflection group of rank $n$ is well-generated, the \textit{Coxeter number h} of $G$ is the largest integer such the there exists an $h$-regular element in the group. These $h$-regular elements are called \textit{Coxeter elements}.

\subsection{Hurwitz Moves}
Given an element $C$ of the CRG $G$ and a reflection factorization $F = (r_1, \ldots, r_{i-1},\hspace{2pt} r_i,\hspace{2pt} r_{i+1}, \ldots, r_n)$ of $c$, we define a \textit{Hurwitz move} at position $i$, where $1\leq i\leq n-1$ to be the following operation:
\[\sigma_i(F) = (r_1, \ldots, r_{i-1},\hspace{12pt} r_{i+1},\hspace{12pt} r_{i+1}^{-1}\hspace{1pt} r_i\hspace{1pt} r_{i+1},\hspace{12pt}r_{i+1}, \ldots, r_n).\] 
Applying a Hurwitz move to a factorization produces a new factorization that multiplies to the same element $c$. If $r_i$ is a reflection in conjugacy class $\mathcal{K}$, then $ r_{i+1}^{-1}r_ir_{i+1}$ is also a reflection in conjugacy class $\mathcal{K}$. 

The \textit{Hurwitz orbit} of such a factorization is the set of all other distinct factorizations that can be reached by applying some number of Hurwitz moves to the original factorization. 

With these facts, we have the following proposition.

\begin{proposition}[{Peterson, {\cite[Prop.~2.2]{Zach}}}]
\label{ZP permutations}
A reflection factorization with a given multiset of conjugacy classes has in its Hurwitz orbit, factorizations with all possible permutations of those conjugacy classes.
\end{proposition}

\subsection{The Group $\boldsymbol{G_6}$}

Out of the $34$ exceptional groups, $6$ of them are real reflection groups (for which the conjecture was proved in \cite{LR}), and $8$ of them are not well generated, so there are $20$ complex reflection groups for which we hope to prove Conjecture \ref{L and R}, two of which have already been proven \cite{Zach}. In this paper we are focusing on the group $G_6$. We define this group by the generators $A$ and $B$ where $A^3 = I = B^2$ and $ABABAB = BABABA$. 

We denote the complex third root of unity as $\zeta = e^{\frac{2 \pi i}{3}}$, and the complex twelfth root of unity as $\gamma = e^{\frac{2 \pi i}{12}}$. More concretely, one can take for  $A$ and $B$ the matrices 
$$A=\begin{pmatrix} 1 &   0       \\
     0 &  \zeta\end{pmatrix} \text{ and  }
B=\frac{1}{3}\begin{pmatrix} \gamma^{11}-\gamma^7 & -2\gamma^{11}-\gamma^7\\
    2\gamma^{11}+4\gamma^7 & \gamma^7-\gamma^{11} \end{pmatrix}.$$ 
 The group $G_6$ has four Coxeter elements, one of which is
$$C = AB = \frac{1}{3}\begin{pmatrix} \gamma^{11}-\gamma^7   &
    -2\gamma^{11}-\gamma^7  \\
    2\gamma^{11}-2\gamma^7 & 2\gamma^{11}+\gamma^7
  \end{pmatrix}.$$

For purposes of explicit calculations, as in the proofs of Propositions~\ref{basic facts} and~\ref{7 base case} below, we computed with the single Coxeter element $C$ mentioned here.  From \cite[Proposition 1.4]{RRS}, this suffices to prove these results for any Coxeter element: one can use an appropriate reflection automorphism to transfer the necessary statements from any Coxeter element to any other.

In our proof it is also helpful for us to consider the CRG $G_4$. This group is defined by two generators, $A'$ and $B'$ where $A'^3 = I = B'^3$ and $A'B'A' = B'A'B'$. Concretely, we can take these generators of $G_4$ to be 
$$A'= A = \begin{pmatrix} 1 &   0       \\
     0 &  \zeta\end{pmatrix} \text{ and  }
B'=\frac{1}{3}\begin{pmatrix} \zeta-\zeta^2 & 3\zeta^2\\
    -2\zeta^2 & -\zeta-2\zeta^2 \end{pmatrix}.$$

\begin{defn}
\label{sub-conjugacy class}
    Suppose that a set of reflections $X$ has a subset of reflections $Y$ that are all in the same conjugacy class $C$. If there are elements in $Y$ that are in different conjugacy classes when only considered in the subgroup generated by $Y$, then these conjugacy classes are called \textit{sub-conjugacy classes}.
\end{defn}

The following proposition gives us a list of basic facts about the complex reflection group $G_6$.

\begin{proposition}
\label{basic facts}
The following are true.

\begin{enumerate}
    \item The complex reflection group $G_6$ has $48$ elements, $14$ of which are reflections. The set of these reflections, which we denote by $\mathcal{R}$, can be split up by conjugacy class into the following three sets of reflections: 
    \begin{align*}
        \mathcal{R}_1 & = \left\{A , ABA(AB)^{-1}, BAB^{-1} , (BA)^{-1}ABA\right\} \\
        \mathcal{R}_2 & = \left\{A^{-1} , ABA^{-1}(AB)^{-1} , BA^{-1}B^{-1} , (BA)^{-1}A^{-1}BA\right\} \\
        \mathcal{S} & = \left\{B, ABA^{-1} , A^{-1}BA ,  (AB)^{-1}BAB , \right. \\ 
        & \hspace{1in} \left. BAB(BA)^{-1}, (A^{-1}BA)^{-1}BA^{-1}BA\right\}
    \end{align*}
    where $\mathcal{R} = \mathcal{R}_1 \cup \mathcal{R}_2 \cup \mathcal{S}$.
    
    \item Let $\mathcal{R}' = \mathcal{R}_1 \cup \mathcal{R}_2$. Then $\mathcal{R}'$ generates the CRG $G_4$. The set $\mathcal{S}$ generates a group isomorphic to $G(4,2,2)$, as defined in our proof of this statement. 
    
    \item  The conjugacy class $\mathcal{S}$ of $G_6$ contains three sub-conjugacy classes, $\mathcal{S}_1$, $\mathcal{S}_2$, and $\mathcal{S}_3$, where
    \begin{gather*} 
    \hspace{10pt}\mathcal{S}_1 = \{B, (A^{-1}BA)^{-1}BA^{-1}BA\},\\
    \hspace{40pt}\mathcal{S}_2 = \{ABA^{-1}, BAB(BA)^{-1}\},\text{ and }\\
    \hspace{10pt}\mathcal{S}_3 = \{A^{-1}BA, (AB)^{-1}BAB\}.
    \end{gather*}
    Each reflection in $\mathcal{S}$ only commutes with reflections in its sub-conjugacy class. 
    
    \item All elements of $\mathcal{R}'$ are of order $3$, and all elements of $\mathcal{S}$ are of order $2.$
    
    \item By the order of the elements and our choice of the element $A$, $\det(A) = \zeta$, $\det(B) = -1$, and $\det(C) = -\zeta$.

    \item Consider the pair $(x,y)$ where $x$ is a reflection of $G_6$ in conjugacy class $\mathcal{R}_1$ and $y$ is a reflection of $G_6$ in conjugacy class $\mathcal{R}_2$. If $x$ and $y$ are inverses, then applying a Hurwitz move to this pair simply commutes the two elements. However, if $x$ and $y$ are not inverses, then we have a length $4$ Hurwitz orbit 
  \[
    \hspace{35pt}(x, y) \overset{\sigma}{\to} (y, x')\overset{\sigma}{\to} (x', y') \overset{\sigma}{\to} (y', x) \overset{\sigma}{\to} (x, y),
  \]
  where two additional elements, $x'$ from conjugacy class $\mathcal{R}_1$ and $y'$ from conjugacy class $\mathcal{R}_2$, are introduced in the orbit.
  
  \item Consider the pair $(x,y)$ where $x$ is a reflection of $G_6$ in conjugacy class $\mathcal{R}_1$ or $\mathcal{R}_2$ and  $y$ is a reflection of $G_6$ in the conjugacy class $\mathcal{S}$. Then, the Hurwitz orbit would be of length $6$
  \begin{multline*}
\hspace{40pt}(x, y) \overset{\sigma}{\to} (y, x')\overset{\sigma}{\to} (x', y') \overset{\sigma}{\to} (y', x'')\\
\overset{\sigma}{\to} (x'', y'') \overset{\sigma}{\to} (y'', x) \overset{\sigma}{\to} (x, y),
\end{multline*}
  where there are four additional elements introduced in the orbit: $x'$ and $x''$ from the same conjugacy class as $x$ and $y'$ and $y''$ from the same conjugacy class as $y$.
  
\end{enumerate}
\end{proposition}
\begin{proof}
First we prove (2).  The group $G(4,2,2)$ belongs to the infinite family $G(m, p, n)$ of finite complex reflection groups (named by Shephard--Todd \cite{ST}); it consists of the sixteen $2 \times 2$ monomial matrices with nonzero entries $\pm i$ and $\pm 1$ such that these nonzero entries multiply to $\pm1$.

To show that $\mathcal{S}$ generates a group isomorphic to $G(4,2,2)$, we use the change of basis matrix
$$M = \frac{1}{2}\begin{pmatrix} 2 &   -\gamma^4-\gamma^7-2\gamma^{11}       \\
     \gamma^4+\gamma^{11} &  -\gamma^4-\gamma^7 \end{pmatrix}.$$
     This conjugates $B$ to $\begin{pmatrix}1 & 0\\ 0 & -1\end{pmatrix}$, $ABA^{-1}$ to $\begin{pmatrix}0 & 1\\ 1 & 0\end{pmatrix}$ and $A^{-1}BA$ to $\begin{pmatrix}0 & -i \\ i & 0\end{pmatrix}$.  These conjugated matrices generate $G(4,2,2)$.
     
All the other parts can be observed by checking the finite instances in which they occur, which we did by calculations on Sage \cite{Sage}.
\end{proof}

\begin{lemma}
\label{ordering of classes}
A reflection factorization with a given multiset of conjugacy classes has in its Hurwitz orbit a factorization in which, when reading the reflection from left to right, first we see all the reflections from $\mathcal{R}_1$, then all the reflections from $\mathcal{R}_2$, and then all the reflections from $\mathcal{S}$.
\end{lemma}
\begin{proof}
This follows direct from Proposition \ref{ZP permutations}, as all permutations of conjugacy classes can be reached through Hurwitz moves.
\end{proof}

\begin{proposition}
  \label{mod 3 conj classes}
  In a reflection factorization of Coxeter element $C$, let $x$ be the number of elements from $\mathcal{R}_1$, $y$ be the number of elements from $\mathcal{R}_2$, and $z$ be the number of elements from $\mathcal{S}$. If we write $x = 3x'+s$ and $y = 3y'+t$ where $x', y', s, t$ are integers such that $0\leq s,t<3$, then
  $1\equiv (s+2t)\pmod 3$ and $z \equiv 1 \pmod 2$.
\end{proposition}

\begin{proof}
    We know by Lemma \ref{ordering of classes} that given a reflection factorization of $C$ we can perform Hurwitz moves to attain a reflection factorization of the form 
    $$(a_1, \ldots, a_x, a'_1, \ldots, a'_y, b_1, \ldots, b_z)$$
    where $a_i\in \mathcal{R}_1$,  $a'_i\in \mathcal{R}_2$, and  $b_i\in \mathcal{S}$. Then, 
    $$\det(C) = \det(a_1)\cdots \det(a_x)\det(a'_1)\cdots \det(a'_y)\det(b_1) \cdots \det(b_z).$$
    The determinant of any reflection in $\mathcal{R}_1$ is $\zeta$, of any reflection in $\mathcal{R}_2$ is $\zeta^2$, and of any reflection in $\mathcal{S}$ is $-1$. Additionally, the determinant of $C$ is $-\zeta$. Thus, the above equation implies
    $$-\zeta = \zeta^x\zeta^{2y}(-1)^z.$$
    We see that we must have $z \equiv 1\pmod 2$. With this, we can then simplify the equation to
    $$\zeta = \zeta^{x+2y} \text{, or } 1 = 3x'+s + 6y'+2t.$$
    This simplifies to $3(x'+2y') = 1-(s+2t)$, or $1\equiv (s+2t)\pmod 3$, as needed. 
\end{proof}

\section{The Proof}

\subsection{Local Results}

Our goal is to find ways to take an length-$\ell$ factorization and relate it to a $\ell-1$ or $\ell-2$ length factorization. This allows us to apply the principle of induction. 

We begin by addressing patterns that arise in reflection factorizations when Hurwitz moves are applied depending on the conjugacy classes of the elements they are applied to.

\begin{defn}
Given a tuple of reflections
$(x_1, \ldots, x_n)$, if we are able to perform Hurwitz moves to get a series of $n$ identical elements, such as $(\ldots, t,t, \ldots, t, \ldots)$, we call this $n$-tuple within the tuple of reflections a \textit{perfect $n$-tuple}. In the case when $n = 2$, we have a \textit{perfect pair}, and when $n = 3$ we have a \textit{perfect triple}.
\end{defn}

We wish to show that if we have some perfect tuple, it may be replaced by a single reflection when preforming Hurwitz moves. This allows us to relate a reflection factorization with one of shorter length so that we may apply the principle of induction.

In $\mathcal{R}_1$ and $\mathcal{R}_2$, all reflections have order $3$, so finding a perfect pair $(t,t)$ in $\mathcal{R}_1$ can be replaced by its square $t^2$ in $\mathcal{R}_2$ and vice versa. In $\mathcal{S}$, all reflections have order $2$, so a perfect pair would multiply to the identity, which is not a reflection. Thus, for these reflection we want a perfect triple $(t,t,t)$ which can be replaced by the original element $t$. Using these relationships we can make the following observations.

\begin{lemma}
  \label{4 pair in 6}
  Given a reflection factorization $T$ in $G_6$ of Coxeter element $C$, if there are at least three elements from the set of reflections $\mathcal{R}'$ represented, then we are able to find a perfect pair $(t,t)$ in the Hurwitz orbit of $T$, for some $t$ in $\mathcal{R}'$.
  
\end{lemma}
\begin{proof}
  Consider a reflection factorization $T$ of a Coxeter element of $G_6$. We know by Lemma \ref{ordering of classes} that we can assume this factorization to be sorted by conjugacy class where the elements from  $\mathcal{R}'$ appear first and elements from $\mathcal{S}$ follow. In the factorization, we let the elements from $\mathcal{R}'$ be comprised of $m$ elements from $\mathcal{R}_1$ and $n$ elements from $\mathcal{R}_2$. First we consider when $m$ or $n$ is zero. That is, the elements from $\mathcal{R}'$ are in fact from a single conjugacy class, $\mathcal{R}_1$ or $\mathcal{R}_2$. 
  
  First consider $m = 0\neq n$. If $n>4$, the result holds by the pigeonhole principle, and we cannot have $n=3$, by Proposition \ref{mod 3 conj classes}. Thus we only need to check the case where $n = 4$, where the factorization is 
  \[(a_1, a_2, a_3, a_4, b_1, \ldots, b_k)\]
  where $a_i \in \mathcal{R}'$, all the $a_i$ are distinct, and $b_i \in \mathcal{S}$. If we perform Hurwitz moves on position $4$ between the initial elements $a_4$ and $b_1$, by Proposition \ref{basic facts}(7) we introduce two new elements from $\mathcal{R}_1$ which by the pigeonhole principle must be two of $a_1, a_2, a_3$, all of which are already in the tuple. Thus, we have two identical elements so we are able to form a $(t,t)$ pair. 
  
  The case $n = 0\neq m$ is identical.
  
  Now we consider both $m$ and $n$ not equal to zero. As $n+m \geq 3$, there must be at least two elements from one conjugacy class. Without loss of generality, let $n\geq 2$. Using Hurwitz moves we rearrange the tuple so that we have
  $$(\ldots, x, y_1, z_1, b_1,\ldots, b_k).$$ 
  If $y_1 = z_1$, we are done. Otherwise, consider $y_1 \neq z_1$. If $x$ and $y_1$ are inverses, we can easily remedy this  by performing a Hurwitz move on the pair $(y_1, z_1)$. So, we assume $x$ and $y_1$ are not inverses. We then perform Hurwitz moves between the pairs $(x, y_1)$ and $(z_1, b_1)$. By parts (6) and (7) of Proposition \ref{basic facts} the orbit of $(x, y_1)$ will produce one additional  element $y_2$ from conjugacy class $\mathcal{R}_2$, and the orbit of $(z_1, b_1)$ will produce two additional elements $z_2, z_3$ from conjugacy class $\mathcal{R}_2$. Therefore we have five elements, $y_1, y_2, z_1, z_2, z_3$ all in the same conjugacy class where there are four elements total. By the pigeonhole principle there is are $i,j$ such that $y_i = z_j$, giving us our $(t,t)$ pair.  
\end{proof}

\begin{lemma}
  \label{422 perfect tuple}
  Given a tuple of elements $\mathcal{S}$ of length $m \geq 3$, if $n<m$ elements of the tuple, $t_1, \ldots, t_n$, are in the same sub-conjugacy class, and if there is at least one other element $s$ in the tuple not in that sub-conjugacy class, then there exists a perfect $n$-tuple. 
\end{lemma}
\begin{proof}
  First, we just consider $t_1, t_2$, and $s$.
  
  If $t_1 = t_2$ we can simply rearrange the tuple using Hurwitz moves so that $t_1$ and $t_2$ are next to each other, giving us a perfect pair.
  
  Otherwise, we consider $t_1 \neq t_2$. As $s$ is not in the same conjugacy class as $t_1$ and $t_2$, we can rearrange the tuple so that we have  
  $$(\ldots, t_1, s', t_2, \ldots)$$
  where $s'$ is the element $s$ after being acted on by the appropriate Hurwtiz moves. We can then perform a Hurwitz move between $s'$ and $t_1$. By Proposition \ref{basic facts}(3), the sub-conjugacy classes of $\mathcal{S}$ have two elements which only commute with each other, so this gives us the tuple $$(\ldots, s', t_2, t_2, \ldots)$$
  with a perfect pair $(t_2, t_2)$. If we wish to have the pair $(t_1, t_1)$, we can take this factorization, move $s'$ to be the rightmost of the three, where it will be conjugated to some element $s''$, and then move it back to the leftmost position, so we have 
  $$(\ldots, s'', t_1, t_1, \ldots)$$
  giving us the perfect pair $(t_1, t_1)$. Thus, given two elements $t_1$ and $t_2$, we can not only find a perfect pair, but we can choose whether it is a pair of $t_1$ or $t_2$. 
  
  We can rewrite either result as $(\ldots, r, t, t, \ldots)$. We then introduce the element $t_3$ and rearrange the tuple using Hurwitz moves to get 
  $$(\ldots, t, t, r', t_3, \ldots).$$
  We apply the same process as before to get 
  $$(\ldots, r'', t, t, t, \ldots).$$
  We can continue introduce all elements of $t_1, \ldots, t_n$ one by one until we have
  $$(\ldots, q, \underbrace{t, t, \ldots, t}_{n}, \ldots)$$
  where $q\in\mathcal{S}$, giving us a perfect $n$-tuple, as needed.
\end{proof}

\begin{corollary}
\label{422 n+1 tuple}
    Let $n$ be a positive integer. For any tuple consisting of elements of $\mathcal{S}$ with length $3n +1$, we are able to find a perfect $(n+1)$-tuple. 
\end{corollary}
\begin{proof}
  When $n = 0$, there is only one element and the result is trivial. Consider $n\geq 1$. In a tuple of $3n+1$ elements, all of which are from one of three sub-conjugacy classes, we have at worst case, without loss of generality $n$ elements in $\mathcal{S}_1$, $n$ elements in $\mathcal{S}_2$, and $n+1$ elements in $\mathcal{S}_3$. By Lemma \ref{422 perfect tuple}, were are able to produce a perfect $(n+1)$-tuple using Hurwitz moves.
\end{proof}

In our proof of Theorem \ref{THE theorem}, to show that any two reflection factorizations of $C$ are in the same orbit, we take a canonical factorization from each orbit and show that all factorizations in that orbit can reach this canonical factorization through Hurwtiz moves. This canonical form is defined by the following.

\begin{defn}
We say that a reflection factorization of the form
$$(\underbrace{A, \ldots, A}_{n}, \underbrace{A^{-1}, \ldots, A^{-1}}_{m}, \underbrace{B, \ldots, B}_{k}).$$
is a \textit{standard factorization} and we denote it by $[n,m,k]$. 

\end{defn}

Given a reflection factorization of $C$ that has $n$ elements in the set $\mathcal{R}_1$, $m$ elements in the set $\mathcal{R}_2$, and $k$ elements in the set $\mathcal{S}$, we wish to show that it is in the same Hurwitz orbit as the standard factorization $[n,m,k]$.

\subsection{Marked Factorizations}

\begin{defn}
[{\cite[Defn.\ 4.17]{Zach}}]
\label{marked element}
A \textit{marked element} in a reflection factorization is an element $t$ that has been marked, denoted as $t^*$. A \textit{marked factorization} $\hat{T}$ is a factorization which contains a marked element.
\end{defn}

Now we need to be able to apply Hurwitz moves on this marked element, so we define a how Hurwitz moves would function on these elements.

\begin{defn}
[{\cite[Defn.\ 4.18]{Zach}}]
A \textit{marked Hurwitz move}, $\sigma^*$, is defined as follows for marked reflection factorizations:
\begin{equation}
    (\ldots, t_i, t_{i+1},\ldots) \xrightarrow{\sigma_i^*} (\ldots, t_{i+1}, t_{i+1}^{-1}\cdot t_i \cdot t_{i+1}, \ldots)
\end{equation}
\begin{equation}
    (\ldots, t_i, t_{i+1}^*) \xrightarrow{\sigma_i^*} (\ldots, t_{i+1}^*, t_{i+1}^{-1}\cdot t_i \cdot t_{i+1},\ldots)
\end{equation}
\begin{equation}
    (\ldots, t_i^*, t_{i+1}, \ldots) \xrightarrow{\sigma_i^*} (\ldots, t_{i+1}, (t_{i+1}^{-1}\cdot t_i \cdot t_{i+1})^*, \ldots)
\end{equation}
\end{defn}

So, a marked Hurwitz move is identical to a Hurwitz move, while also shifting the position of the marking $^*$. Thus, all previously made statements about Hurwitz moves are also true for the marked Hurwitz move.

In our inductive proof of the theorem, we do one of the following. 
\begin{enumerate}
    \item Take a perfect pair $(t,t)$ in $\mathcal{R}_1$  and replace it with the marked element $(t^2)^*$.
    \item  Take a perfect pair $(t,t)$ in $\mathcal{R}_2$  and replace it with the marked element $(t^2)^*$.
    \item Take a perfect triple $(t,t,t)$ in $\mathcal{S}$  and replace it with the marked element $(t^3)^* = t^*$.
\end{enumerate}

 By Lemma \ref{4 pair in 6} and Corollary \ref{422 n+1 tuple} this will be possible whenever our original factorization has at least $3$ reflections in $\mathcal{R}'$ or $7$ reflection in $\mathcal{S}$. By Proposition \ref{mod 3 conj classes}, we can see that we are guaranteed one of these cases for all factorization of length $8$ and greater. For factorization of length $7$ or less, we check the theorem using Sage, and these factorizations will function as our base case. 

\begin{proposition}
\label{7 base case}
  The conjecture is true for reflection factorizations up to and including length $7$.    
\end{proposition}
\begin{proof}As the number of elements here is finite, we are able to perform a finite number of calculations on Sage to prove the statement. Given a length $\ell$ factorization, we check with Sage how many possible reflection factorizations there are in total. Then, we consider the possible standard forms $[n,m,k]$, as limited by Proposition \ref{mod 3 conj classes}. Using Sage, we then computed the length of the orbits of each possible standard form. In all cases, the sum of these sizes is equal to the total number of possible reflection factorizations, so the conjecture holds for length $\ell$. We verified this fact for $2\leq \ell\leq 7$. 
\end{proof}

To prove our theorem inductively, we must show that performing Hurwitz moves at length $n-k$ factorizations have the same results as performing certain Hurwitz moves on the appropriate  length $n$ factorization.

\begin{lemma}
  \label{moving marked elements}
  Let $T$ be a reflection factorization $T = (\ldots, t, t, \ldots, t, \ldots)$ with a perfect $n$-tuple $(t,t,\ldots,t)$, and let $\hat{T}$ be the marked factorization resulting from letting  $(t,t,\ldots,t) = (t^n)^*$, so that we have $\hat{T} = (\ldots, (t^n)^*, \ldots)$. 
  
  Suppose that the marked factorization $\hat{S}$ is obtained from $\hat{T}$ by performing some Hurwitz moves, and that $\hat{S}$ has marked element $(s^n)^*$. Suppose further that this marked element has a unique expansion 
  \[s^n =\underbrace{s\cdot \cdots \cdots s}_{n}\]
   and let $S$ be the factorization that we get by replacing $(s^n)^*$ with the $n$-tuple $(s, \ldots, s)$. Then, there is a series of Hurwitz moves that can be performed on $T$ to obtain the factorization $S$. 
  
\end{lemma}
\begin{proof}
  By induction, it suffices to prove the case where $\hat{T}$ is related to $\hat{S}$ by a single Hurwitz move $\sigma_i$. If this Hurwitz move does not involve the marked element, then the effect of performing a Hurwitz move in $\hat{T}$ is identical to performing a Hurwitz move in $T$ and so the result follows immediately in this case. Otherwise, we first consider the following move $$
 \hat{T} = (\ldots ,s, (t^n)^*, \ldots) \xrightarrow{\sigma_i^*} (\ldots ,(t^n)^*, t^{-n}\cdot s\cdot t^n, \ldots)= : \hat{S}.
 $$
  Thus, $T = (\ldots, s, t,\ldots, t, \ldots)$ and $S = (\ldots, t, \ldots, t, t^{-n}\cdot s\cdot t^n, \ldots)$. We can then perform the following Hurwitz moves on $T$: 
  \begin{equation*}
      \begin{split}
          T = (\ldots, s, t, \ldots, t, \ldots) & \xrightarrow{\sigma_i} (\ldots, t, t^{-1}\cdot s\cdot t, t, \ldots, t, \ldots)\\
                                                & \xrightarrow{\sigma_{i+1}} (\ldots, t, t, t^{-2}\cdot s\cdot t^2, t, \ldots, t, \ldots) \\
                                                & \vdots \\
                                                & \xrightarrow{\sigma_{i+n-1}} (\ldots, t, \ldots, t, t^{-n}\cdot s\cdot t^n, \ldots) = S.
      \end{split}
  \end{equation*}
  as needed. We now consider the move 
  $$\hat{T} = (\ldots , (t^n)^*,s,  \ldots)  \xrightarrow{\sigma_i^*} (\ldots ,s, (s^{-1}\cdot t^n \cdot s\cdot)^* , \ldots):= \hat{S}.$$
  Thus, $T = (\ldots, t, \ldots, t, s, \ldots)$ and $S = (\ldots, s, s^{-1}\cdot t \cdot s,\ldots , s^{-1}\cdot t \cdot s, \ldots)$. We can then perform the following Hurwitz moves on $T$:
  \begin{equation*}
      \begin{split}
          T = (\ldots, t, \ldots, t, s, \ldots) & \xrightarrow{\sigma_{i+n}} (\ldots, t, \ldots, t, s, s^{-1}\cdot t \cdot s,\ldots)\\
                                                & \xrightarrow{\sigma_{i+n-1}} (\ldots, t, \ldots, t, s, s^{-1}\cdot t \cdot s, s^{-1}\cdot t \cdot s,\ldots)\\
                                                & \vdots\\
                                                & \xrightarrow{\sigma_{i+1}} (\ldots, s, s^{-1}\cdot t\cdot s, \ldots, s^{-1}\cdot t\cdot s, \ldots),
      \end{split}
  \end{equation*} 
  as needed.
\end{proof}

\subsection{Proof of the Theorem}

We now have all the facts need to prove our main theorem by induction on the length of the reflection factorizations of $C$. We restate it here for convenience. 

\begin{recap}
Let $T$ and $T'$ be two length $n$ reflection factorizations of a Coxeter element of complex reflection group $G_6$. Then, $T$ and $T'$ are in the same Hurwitz orbit if and only if they have the same multiset of conjugacy classes.
\end{recap}
\begin{proof}

We know that Hurwitz moves preserve conjugacy class, so the forward direction holds. 

Now we consider the backward direction which we prove by induction on the length $p$ of the reflection factorizations $T$ and $T'$. We have our base cases where $p\leq 7$ checked in Proposition \ref{7 base case}. 
Assume that the theorem holds for all factorizations of length $\ell$ or less.  Consider a factorization $T$ of length $\ell + 1$. We wish to show that $T$ is in the same Hurwitz orbit as the standard factorization with the same multiset of conjugacy classes, $[n,m,k]$.

If $T$ has $7$ or more elements from the conjugacy class $\mathcal{S}$, then by Proposition \ref{422 n+1 tuple} we are able to find a perfect triple $(t,t,t)$ in the Hurwitz orbit of $T$. From this perfect triple we are able to create the marked element $(t^3)^* = t^*$ giving us the marked factorization $\hat{T}$. As $\hat{T}$ is of length $\ell-1$, by the inductive hypothesis it is in the same Hurwitz orbit as the standard factorization $[n,m,k-2]$. Thus, the marked element here is the element $B^*$, which when replaced with $(B,B,B)$ gives us the standard factorization $[n,m,k]$. By Lemma \ref{moving marked elements} this means that the standard factorization is in the same Hurwitz orbit as $T$. 

If $T$ has fewer than $7$ elements from $\mathcal{S}$, then by Proposition \ref{mod 3 conj classes} there must be at least $3$ elements from $\mathcal{R}'$ as we are considering $\ell\geq 7$. By Lemma \ref{4 pair in 6} we know that we are able to find a perfect pair $(t^2,t^2)$ in a factorization in the Hurwitz orbit of $T$ so without loss of generality, we may assume that such a pair already exists in $T$. From this perfect pair we are able to create the marked element $(t^4)^* = t^*$ giving us the marked factorization $\hat{T}$. As $\hat{T}$ has length $\ell$, by the inductive hypothesis it is in the same Hurwitz orbit as the standard factorization with the corresponding multiset of conjugacy classes. The marked element would then either be the element $(A^{-1})^*$ or $A^*$. If we have $(A^{-1})^*$,
we have 
\[(\underbrace{A, \ldots, A,}_{n-2} \underbrace{A^{-1}, \ldots, (A^{-1})^*, \ldots, A^{-1},}_{m+1} \underbrace{B, \ldots, B}_{k}).\]
We expand the marked element to $(A,A)$ 
\[(\underbrace{A, \ldots, A,}_{n-2} \underbrace{A^{-1}, \ldots, A,A, \ldots, A^{-1},}_{m+2} \underbrace{B, \ldots, B}_{k})),\]
where these two elements can be easily moved to the $\mathcal{R}_1$ section of the factorization as $A$ and $A^{-1}$ commute under Hurwitz moves. . For the same reasons, if we have $A^*$, we have
\[(\underbrace{A, \ldots,A^*, \ldots, A,}_{n+1} \underbrace{A^{-1}, \ldots, A^{-1}}_{m-2}, \underbrace{B, \ldots, B}_{k}).\]
We expand the marked element to to $(A^{-1},A^{-1})$
\[(\underbrace{A, \ldots,A^{-1}, A^{-1}, \ldots, A,}_{n+2} \underbrace{A^{-1}, \ldots, A^{-1}}_{m-2}, \underbrace{B, \ldots, B}_{k}),\]
 which can be easily moved to the $\mathcal{R}_2$ section of the factorization if it is not already. These expansions give us a $[n,m,k]$ with length $\ell+1$. By Lemma \ref{moving marked elements} this means that the standard factorization is in the same Hurwitz orbit as $T$. As $T$ is arbitrary, and $T'$ has the same multiset of conjugacy classes as $T$, then $T'$ is also in the same Hurwitz orbit as $[n,m,k]$. Thus, $T$ and $T'$ are in the same Hurwitz orbit, as needed. 
\end{proof}

\end{document}